\documentclass[12pt]{article}
\usepackage{amsfonts}
\usepackage{graphicx}
\usepackage{tabularx}
\usepackage{array}
\usepackage[usenames,dvipsnames]{color}
\usepackage{comment}
\usepackage{amsmath}
\usepackage{amsthm}
\usepackage{amssymb}
\usepackage{fullpage}
\usepackage[dvipsnames]{xcolor}
\usepackage{tikz}
\usepackage{listings}

\DeclareMathOperator{\im}{Im}

\begin{document}

\def\PB{\color{purple}}
\def\RB{\color{blue}}
\def \red {\textcolor{red} }

\newcommand{\spec}{\hat{p}}
\newcommand{\mb}{\mathbf}

\newtheorem{theorem}{Theorem}[section]
\newtheorem{proposition}[theorem]{Proposition}
\newtheorem{observation}[theorem]{Observation}
\newtheorem{assumption}[theorem]{Assumption}
\newtheorem{conjecture}[theorem]{Conjecture}
\newtheorem{question}[theorem]{Question}
\newtheorem{claim}[theorem]{Claim}
\newtheorem{lemma}[theorem]{Lemma}
\newtheorem{corollary}[theorem]{Corollary}
\theoremstyle{definition}
\newtheorem{definition}[theorem]{Definition}
\newtheorem{construction}[theorem]{Construction}
\newtheorem{example}[theorem]{Example}
\newtheorem{remark}[theorem]{Remark}
\newcommand{\s}[1]{\left(#1\right)}

\def\epsilon{\varepsilon}
\def\st{\colon}
\def\cat{\colon\!}
\def\FF{{\mathbb F}}
\def\NN{{\mathbb N}}
\def\VV{{\mathbb V}}
\def\UU{{\mathbb U}}
\def\ZZ{{\mathbb Z}}
\def\brm#1{{\rm #1}}
\def\to{\rightarrow}
\def\CL#1{\left\lceil #1\right\rceil}
\def\FL#1{\left\lfloor #1\right\rfloor}
\def\C#1{\left| #1\right|}
\def\VEC#1#2#3{#1_{#2},\ldots,#1_{#3}}
\def\VECOP#1#2#3#4{#1_{#2}#4\cdots#4#1_{#3}}
\def\SE#1#2#3{\sum_{#1=#2}^{#3}}
\def\UE#1#2#3{\bigcup_{#1=#2}^{#3}}
\def\esub{\subseteq}
\def\ang#1{\langle #1\rangle}
\def\nul{\varnothing}
\def\cC{{\mathcal C}}
\def\cB{{\mathcal B}}
\def\bI{{\bf I}}
\def\bT{{\bf T}}
\def\la{\langle}
\def\ra{\rangle}
\def\ov{\overline}
\long\def\skipit#1{}

\title{Strong parity edge-colorings of graphs}

\author{
Peter Bradshaw\thanks{University of Illinois; {\tt pb38@illinois.edu.}
Research supported by NSF RTG grant DMS-1937241.}\,,
Sergey Norin\thanks{McGill University; {\tt snorin@math.mcgill.ca.}
Research supported by an NSERC Discovery Grant.}\,,
Douglas B. West\thanks{Zhejiang Normal University and University of
Illinois; {\tt dwest@illinois.edu.}  Research supported by National Natural
Science Foundation of China grants NSFC 11871439, 11971439, and U20A2068.}
}

\maketitle
\begin{abstract}
An \emph{edge-coloring} of a graph $G$ assigns a color to each edge of $G$.
An edge-coloring is a \emph{parity edge-coloring} if for each path $P$ in $G$,
it uses some color on an odd number of edges in $P$.  It is a
\emph{strong parity edge-coloring} if for every open walk $W$ in $G$, it
uses some color an odd number of times along $W$.  The minimum numbers of
colors in parity and strong parity edge-colorings of $G$ are denoted
$p(G)$ and $\spec(G)$, respectively.

We characterize strong parity edge-colorings and use this characterization to
prove lower bounds on $\spec(G)$ and answer several questions of Bunde, Milans,
West, and Wu.  The applications are as follows.
(1)~We prove the conjecture that $\spec(K_{s,t})=s \circ t$, where $s \circ t$
is the Hopf--Stiefel function.
(2)~We show that $\spec(G)$ for a connected $n$-vertex graph $G$ equals the
known lower bound $\lceil \log_2 n \rceil$ if and only if $G$ is a subgraph of
the hypercube $Q_{\lceil \log_2 n \rceil }$.
(3)~We asymptotically compute $\spec(G)$ when $G$ is the $\ell$th
distance-power of a path, proving $\spec(P_n^\ell)\sim\ell\CL{\log_2 n}$.
(4)~We disprove the conjecture that $\spec(G)=p(G)$ when $G$ is bipartite by
constructing bipartite graphs $G$ such that $\spec(G)/p(G)$ is arbitrarily
large; in particular, with $\spec(G)\ge\frac{1-o(1)}3 k\ln k$ and
$p(G)\le2k+k^{1/3}$.
\end{abstract}
\section{Introduction}
We study special edge-colorings of graphs, where an \emph{edge-coloring} of a
graph assigns a color to each edge.  Given an edge-coloring, a path $P$ 
is a \emph{parity path} if each color appears an even number of times on $P$.
Similarly, a walk $W$ is a \emph{parity walk} if each color appears an even
number of times along a traversal of $W$.
A \emph{parity edge-coloring (pec)} of a graph $G$ is an edge-coloring that
admits no parity path.  Similarly, a \emph{strong parity edge-coloring (spec)}
of $G$ is an edge-coloring that admits no open parity walk (a walk is
\emph{open} when its endpoints are distinct; otherwise it is \emph{closed}).
The \emph{parity edge chromatic number}, written $p(G)$,
is the minimum number of colors in a parity edge-coloring of $G$,
and the \emph{strong parity edge chromatic number}, written $\spec(G)$,
is the minimum number of colors in a strong parity edge-coloring of $G$.
Note that every spec is also a pec.  Also every pec is a proper edge-coloring
(meaning that incident edges have distinct colors), since incident edges with
the same color would form a parity path.  Hence every graph $G$ satisfies
\[
\chi'(G) \leq p(G) \leq \spec(G),
\]
where the \emph{edge-chromatic number} $\chi'(G)$ is the minimum number of
colors in a proper edge-coloring.  The notions of pec and spec were introduced
by Bunde, Milans, West, and Wu~\cite{BMWW, BMWW2}.

Every pec of a graph $G$ is a \emph{nonrepetitive edge-coloring}, which is an
edge-coloring $\phi$ such that
$(\phi(e_1),\dots,\phi(e_k)) \ne (\phi(e_{k+1}),\dots,\phi(e_{2k}))$
whenever $e_1,\ldots,e_{2k}$ in order form a path with even length.
The special case where $G$ is a path was introduced by 
Thue~\cite{Thue} under the name \emph{square-free words}.
Alon, Grytczuk, Ha\l uszczak, and
Riordan~\cite{AGHR} later extended the concept to general graphs. 
Since a nonrepetitive edge-coloring is a proper edge-coloring,
the minimum number of colors in a nonrepetitive edge-coloring of a graph
$G$ is between $\chi'(G)$ and $p(G)$.
Rosenfeld \cite{Ros} showed that if $G$ is a graph with maximum degree
$D$, then $G$ has a nonrepetitive edge-coloring using at most $D^2 + O(D^{5/3})$
colors; in contrast, $p(G)$ must grow with the number of vertices when $G$
is connected.

In particular, it was proved in~\cite{BMWW} that
$\spec(G) \geq p(G) \geq  \lceil \lg n \rceil$
for every $n$-vertex connected graph $G$ (we write $\lg$ for $\log_2$).
This result was motivated by the study of embeddings of graphs in hypercubes.
Havel and Mor\'avek \cite{HM} characterized connected subgraphs of
hypercubes.  In our language, they proved that a connected graph $G$ is a
subgraph of the $k$-dimensional hypercube $Q_k$ if and only if $G$ has a pec
$\phi$ using at most $k$ colors such that for each cycle $C$ in $G$, no color
appears in $C$ an odd number of times.  Since the natural edge-coloring of
$Q_k$ using color $j$ on edges whose endpoints are binary $k$-tuples differing
in coordinate $j$ is a spec, every subgraph $G$ of $Q_k$ satisfies
$\spec(G)\leq k$.  As a corollary, a tree $T$ embeds in $Q_k$ if and only if
$p(T) \leq k$.  Monotonicity of the parameters and this fact about spanning
trees yield the following:
\begin{proposition}\label{prop:subQk}
For every $n$-vertex connected graph $G$,
\[
\spec(G)\ge p(G)\ge\lceil\lg n\rceil.
\]
In addition, if $G$ is an $n$-vertex connected subgraph of $Q_k$ and
$2^{k-1}<n\le2^k$, then equality holds throughout.
\end{proposition}

Proposition \ref{prop:subQk} suggests the following natural question:
\begin{question}[\cite{BMWW}]\label{q:log}
Which $n$-vertex connected graphs $G$ satisfy $\spec(G)=p(G)=\lceil\lg n\rceil$?
\end{question}

One of our results in this paper answers Question \ref{q:log} by showing that
the sufficient condition in Proposition \ref{prop:subQk} is also necessary.
That is:

\begin{theorem}\label{lgn}
If $G$ is a connected graph on $n$ vertices, then 
$\spec(G) = \lceil \lg n \rceil$ if and only if $G$ is a subgraph of the hypercube $Q_{\lceil \lg n \rceil}$.
\end{theorem}

Our proof of Theorem~\ref{lgn} uses an algebraic characterization of specs that
is our main result and has other applications.  
For clarity, we use the term \emph{edge-coloring} for functions defined on
the edges of a graph and use \emph{labeling} for functions defined on the
vertices.  A \emph{binary labeling} of a graph $G$ is an injective function
$f$ of the form $f\st V(G)\to \VV$, where $\VV$ is a binary vector space
(that is, over $\FF_2$). 
A \emph{canonical edge-coloring} of $G$ is an edge-coloring $\phi$ that arises
from some binary labeling $f$ of $G$ by setting $\phi(uv)=f(u)+f(v)$ for each
edge $uv\in E(G)$, where the addition is vector addition in $\VV$.  Under a
slightly more restricted notion of canonical edge-coloring,
Bunde, Milans, West, and Wu \cite{BMWW2} observed the following:

\begin{observation}[\cite{BMWW2}]\label{obs:canon}
Every canonical edge-coloring is a spec.
\end{observation}
\begin{proof}
Let $\phi$ be a canonical edge-coloring, and let $f$ be an associated binary
labeling that generates $\phi$.  Since every color appears an even number of
times along a parity walk $W$ from $u$ to $v$, the sum of the colors of the
edges along a traversal of $W$ is $0$.  On the other hand, when $y$ is an
internal vertex of $W$ the contribution of $f(y)$ to the colors on the edges
entering and leaving $y$ along $W$ cancels in the sum, so the sum of the colors
along $W$ equals $f(u)+f(v)$.  Since $f$ is injective, $f(u)+f(v)\ne 0$.  Hence
there is no open parity walk.
\end{proof}

To characterize specs, we consider refinements of canonical edge-colorings.
An edge-coloring $\phi$ is a \emph{refinement} of an edge-coloring $\phi^*$ if
$\phi$ is obtained from $\phi^*$ by splitting color classes into smaller sets.
In particular, if $\pi$ is the map that combines sets of colors under $\phi$
into colors under $\phi^*$, then $\phi^*$ is the composition $\pi\circ \phi$.
Every refinement $\phi$ of a spec $\phi^*$ is also a spec, since when an odd
number of occurrences of a color (on an open walk) are grouped into subclasses,
an odd number of those subclasses must have odd size.  Our main result is the
following theorem, which is a partial converse of Observation~\ref{obs:canon}.

\begin{theorem} \label{thm:canonical_intro}\label{main}
Every spec $\phi$ of a connected graph $G$ is isomorphic to a refinement of a
canonical edge-coloring.  That is, there is a canonical edge-coloring $\phi^*$
and a renaming $\pi$ of colors such that $\phi^*=\pi\circ\phi$.
In particular, $\spec(G)$ is the minimum number of colors in a canonical
edge-coloring of $G$.
\end{theorem}

\noindent
The renaming of colors when expressing a spec $\phi$ as a refinement of a
canonical edge-coloring $\phi^*$ need not be injective; multiple colors
under $\phi$ may collapse to a single color under $\phi^*$.

An {\it optimal spec} of a graph $G$ is a spec using $\spec(G)$ colors.
By Theorem~\ref{thm:canonical_intro} and the use of binary vector spaces
in canonical edge-colorings, finding an optimal spec is equivalent to the
problem of mapping $V(G)$ injectively into the vertex set of a hypercube to
minimize the number of difference vectors along edges.

Not having Theorem~\ref{thm:canonical_intro}, Bunde, Milans, West, and
Wu~\cite{BMWW2} determined $\spec(K_n)$ directly.

\begin{theorem}[\cite{BMWW2}]\label{thm:Kn}
$\spec(K_n)=2^{\CL{\lg n}}-1$.  In addition,
every optimal spec of $K_n$ is obtained by deleting $2^{\CL{\lg n}}-n$ vertices
from the canonical edge-coloring of $K_{2^{\CL{\lg n}}}$ obtained by labeling
the vertices with the binary $\CL{\lg n}$-tuples.
\end{theorem}

A result similar to a piece of the argument appeared independently a few years
earlier.  Keevash and Sudakov~\cite{KS} proved that if $K_n$ with $n\ge4$ has
an edge-coloring in which the edges of any copy of $K_4$ receive three colors
or six colors, then $n$ is a power of $2$ (presented also in~\cite{Boll}).
The argument in~\cite{BMWW2} shows that if an optimal spec of a complete
graph fails a similar local condition, then another vertex can be added without
needing another color, up to the next power of $2$.

Let $K_{s,t}$ denote the complete bipartite graph with parts of sizes $s$ and
$t$.  We will use Theorem~\ref{main} to determine $\spec(K_{s,t})$, whose
value was conjectured in~\cite{BMWW2}.  It is the value $s\circ t$ known as
the {\it Hopf--Stiefel function} defined for positive integers $s$ and $t$ by
Hopf~\cite{Hopf} and Stiefel~\cite{Stie} to be the least integer $n$ such that
$(x+y)^n$ is contained in the ideal over $\mathbb F_2[x,y]$ generated by $x^s$
and $y^t$.  Equivalently, $s \circ t$ is the least integer $n$ such that
$\binom{n}{k}$ is even whenever $n-t < k < s$.  Plagne~\cite{Pla} obtained the
following formula for $s \circ t$, later reproved by K\'arolyi~\cite{Kar} with
a short inductive proof (see also~\cite{EK2}):

\begin{theorem}[\cite{Pla,Kar}]
For positive integers $s$ and $t$, 
\[
s \circ t = \min_{j \in \mathbb N} \left \{ 2^j
\left( \CL{\frac{s}{2^j}}+ \CL{\frac{t}{2^j}}- 1 \right ) \right \}
\]
\end{theorem}

\noindent
Note in particular that $n\circ n=2^{\CL{\lg n}}$, since the minimum then occurs
by setting $j=\CL{\lg n}$.

We prove the conjecture of~\cite{BMWW2} that $\spec(K_{s,t})=s\circ t$
using Theorem~\ref{main} and a combinatorial interpretation of $s\circ t$ due
to Yuzvinsky~\cite{Yuz}.  Let $\UU$ denote the vector space over $\FF_2$
consisting of binary sequences having a finite number of $1$s.
Yuzvinsky~\cite{Yuz} interpreted $s\circ t$ as the solution to an extremal
problem.

\begin{theorem}[\cite{Yuz}]\label{thm:circ}
For nonempty subsets $A$ and $B$ of $\UU$, let $A+B=\{a+b\st a\in A, b\in B\}$.
If $|A| = s$ and $|B| = t$, then $|A+B| \geq s \circ t$, and for each pair
$(s,t)$ there exist subsets $A$ and $B$ of $\UU$ achieving equality.
\end{theorem}

By choosing a basis, any finite-dimensional binary vector space embeds
isomorphically in $\UU$.  Also, for any binary labeling $f$ of $K_{s,t}$ that
assigns sets $A$ and $B$ to the parts, the set of colors in the associated
canonical edge-coloring is precisely the sum set $A+B$.  To ensure that $f$ is
injective, we can add a new coordinate that is $0$ for each member of $A$ and
$1$ for each member of $B$.  The vectors assigned to the edges as colors then
all have $1$ in the new coordinate, but the multiplicities of colors along any
open walk do not change.

Therefore, Theorem~\ref{thm:circ} implies that the minimum number of colors in
a canonical edge-coloring of $K_{s,t}$ is $s\circ t$.  The conjecture then
follows immediately from Theorem~\ref{main}:

\begin{theorem}\label{conj:Kst}\label{thm:Kst}
If $s,t\in\NN$, then $\spec(K_{s,t}) = s \circ t$.
In particular, $\spec(K_{n,n})=2^{\CL{\lg n}}$.
\end{theorem}

A few special cases of Theorem~\ref{thm:Kst} were known previously.  The case
$s=t=2^k$ was proved in~\cite{BMWW2}, and Hu and Chang~\cite{HC} proved it when
$s\ge3$ and one of $\{t,t+1,t+2\}$ is a multiple of $2^{\CL{\lg s}}$.
Also, since $\spec(K_{n,n})=2^{\CL{\lg n}}$ in the special case $s=t=n$,
Theorem~\ref{thm:Kst} implies the main result $\spec(K_n)=2^{\CL{\lg n}}-1$ 
of~\cite{BMWW2}, due to the following:

\begin{proposition}[\cite{BMWW}]\label{bip2cliq}
$\spec(K_n)\ge\spec(K_{n,n})-1$.
\end{proposition}
\begin{proof}
From a spec $\phi$ of $K_n$ with vertices $\VEC v1n$, define an edge-coloring
$\phi'$ of $K_{n,n}$ with parts $\VEC x1n$ and $\VEC y1n$ by letting
$\phi'(x_iy_j)=\phi'(x_jy_i)=\phi(v_iv_j)$ when $i\ne j$ and using one new
color on all the edges of the form $x_iy_i$.  Because any walk of even length
in $K_{n,n}$ has both endpoints on the same side of the bipartition, any open
parity walk under $\phi'$ lifts to an open parity walk under $\phi$.
\end{proof}

Our characterization of specs allows us to derive new lower bounds
for $\spec(G)$ based on the algebraic aspects of canonical edge-colorings.
Without describing the technical general lower bounds here, we mention
several applications.

Using Proposition~\ref{prop:subQk} and an inductively constructed spec, for the
$n$-vertex path $P_n$ it is easy to show $p(P_n)=\spec(P_n)=\CL{\lg n}$
(\cite{BMWW}).  More generally, we compute $\spec$ asymptotically for the 
{\it path-power} $P_n^\ell$ with vertices $\VEC v0{n-1}$ in which $v_i$ and
$v_j$ are adjacent if and only if $|i-j|\le \ell$.  The dependence on $n$ in
the upper and lower bounds is the same.

\begin{theorem}\label{pathpow}
For integers $n$ and $\ell$ with $1\le\ell\le\CL{\lg n}$,
$$
\ell\CL{\lg n}-\binom{\ell+1}2<\spec(P_n^\ell)
<\ell\CL{\lg n}-\ell(\FL{\lg \ell}-1).
$$
\end{theorem}

\noindent
When $\ell=1$ the two bounds imply the exact value $\CL{\lg n}$.
The lower bound is based on further analysis of canonical edge-colorings.
The upper bound uses a spec constructed from the well-known binary reflected
Gray code (see~\cite{Gil,RND}, for example).  In fact, this is the canonical
edge-coloring generated by the same labeling as the natural optimal spec
for $P_n$.

The upper bound for $\spec(P_n^\ell)$ is used in our final result about the
relationship between $p(G)$ and $\spec(G)$.  For the $n$-vertex cycle $C_n$,
\cite{BMWW} showed $p(C_n) = \spec(C_n) = \lceil \lg n \rceil$ for even $n$ and
$p(C_n) = \spec(C_n) - 1 = \lceil \lg n \rceil $ for odd $n$.
The fact $\spec(C_n)>p(C_n)$ for odd $n$ suggested asking how large the
difference can be.  Also, all examples with $\spec(G)>p(G)$ found by Bunde,
Milans, West, and Wu were nonbipartite.  This motivated a conjecture and a
question.
\begin{conjecture}[\cite{BMWW}]
\label{conj:bip}
    $p(G) = \spec(G)$ for every bipartite graph $G$.
\end{conjecture}
\begin{question}[\cite{BMWW}]
\label{q:diff}
    What is the maximum of $\spec(G)$ when $p(G) = q$?
\end{question}

Using the algebraic methods from our proof of
Theorem~\ref{thm:canonical_intro}, we construct for each positive integer $k$
that is a power of $2$ a bipartite graph $G$ that satisfies $p(G)=k+1$ and
$\spec(G)>k+\Omega((\lg k)^2)$ (in fact, $\spec(G)=(2-o(1))k$, but we omit
those details).
This disproves Conjecture~\ref{conj:bip}.  We then give a more intricate
construction involving a bipartite version of path-powers to 
give a partial answer to Question~\ref{q:diff}:
\begin{theorem}
For each integer $k$ there is a bipartite graph $G_k$ such that
${p(G_k)\le 2k+o(k)}$ and $\spec(G_k)\ge \s{\frac13-o(1)} k\ln k$.
\end{theorem}

Many questions about $p(G)$ remain open, such as whether it equals $\spec(G)$
when $G$ is $K_n$ or $K_{s,t}$.  We mention one question solely about
$\spec(G)$ that may be of interest.

\begin{question}
What is the maximum of $\spec(G)$ over the family of $n$-vertex planar graphs
with maximum degree $D$?
\end{question}

\section{An algebraic framework}
\label{sec:algebraic}
In this section, we establish an algebraic framework for specs of connected
graphs.  Bunde, Milans, West, and Wu \cite{BMWW2} used a similar framework to
discuss specs of complete graphs, but our discussion applies to all connected
graphs.

When discussing the properties of specs throughout this section, we will
maintain the context of a given connected graph $G$ with a root vertex
$r \in V(G)$, and with respect to this we will discuss various aspects of a
spec of $G$.  With this convention, we will not reintroduce $G$ or $r$ with
every statement or proof.
We write $\langle E(G) \rangle$ for the binary vector space with basis $E(G)$.
Each element of $\langle E(G) \rangle$ corresponds to an edge subset of $G$.

Most discussions of edge-coloring view the colors as natural numbers, but
normally we do not view $\NN$ as a binary vector space.  To relate
edge-colorings using positive integers to canonical edge-colorings, we encode
integer colors as special vectors in the binary vector space $\UU$.  Recall
that $\UU$ is the set of binary sequences with finitely many $1$s; $\UU$ is a
binary vector space under component-wise binary addition.  We write the
additive identity (the all-$0$ sequence) as $0$.

\begin{definition}
The \emph{weight} of a vector in $\UU$ is the number of coordinates that
have value $1$.  An \emph{atom} in $\UU$ is a vector of weight $1$.
An \emph{atomic edge-coloring} is an edge-coloring in which each edge is
assigned an atom in $\UU$ as its color. 
\end{definition}

Given an edge-coloring using positive integers, we can encode each integer
color $c$ as the atom in $\UU$ with $1$ in coordinate $c$, converting the
edge-coloring to an atomic edge-coloring.  Taking the binary sum of two vectors
in $\UU$ then corresponds to taking the symmetric difference of the
corresponding subsets of colors in $\NN$.  The encoding does not change color
classes or affect whether the coloring is a spec.  Thus we can discuss any
edge-coloring by positive integers in its equivalent atomic form.

Henceforth when studying an edge-coloring $\phi$ in this section, we will
assume that $\phi$ is given as an atomic edge-coloring.  Our main goal when
$\phi$ is also a spec is to obtain an associated canonical edge-coloring using
colors in a quotient space of $\UU$, without increasing the number of colors.

We extend an edge-coloring $\phi$ of $G$ to a linear mapping from $\ang{E(G)}$ to
$\UU$ using binary vector sums by setting $\phi(A)=\sum_{e\in A}\phi(e)$ for
each edge subset $A\esub E(G)$.  Furthermore, when $W$ is a walk in $G$ with
edges $(\VEC e1t)$ in order, we define $\phi(W)=\phi(e_1)+\cdots+\phi(e_t)$.
The contribution from $\phi(e)$ vanishes if $e$ appears with even multiplicity
along $W$.  Our algebraic setting is motivated by the following observation,
which encodes the definition that parity walks are those using each color an
even number of times.

\begin{observation}
A walk $W$ is a parity walk for an edge-coloring $\phi$ if and only if 
$\phi(W)=0$.
\end{observation}

In this language, the definition becomes that $\phi$ is a spec if and only if
the nullspace of $\phi$ is contained in the subspace of $\ang{E(G)}$ consisting
of $\{\phi(W)\st W \textrm{ is a closed walk in $G$} \}$.  This subspace
contains the images of other sets that are not walks, such as unions of
disjoint cycles, but such sets can be augmented to closed walks through $r$
that have the same image under the linear map $\phi$.  This leads to a useful
characterization of specs in terms of $\UU$.

\begin{definition}
Let $W\cat Z$ denote the concatenation of walks $W$ and $Z$ ($W$ followed by
$Z$), and let $\ov W$ denote the reverse of $W$.  Under an edge-coloring
$\phi$ of $G$, for each $v \in V(G)$ let
\[S_v=
\{\phi(W)\st W \textrm{ is a walk from $r$ to $v$}    \}.
\]
\end{definition}

\begin{lemma}
\label{lem:disjoint}
An edge-coloring $\phi$ of $G$ is a spec if and only if
$S_u \cap S_v = \nul$ for any distinct vertices $u,v \in V(G)$.
\end{lemma}
\begin{proof}
Given an edge-coloring $\phi$ of $G$,
we prove the contrapositive of each implication.

If $S_u\cap S_v\ne\nul$ for distinct vertices $u$ and $v$, then there exist
an $r,u$-walk $W$ and an $r,v$-walk $W'$ such that $\phi(W)=\phi(W')$.  Since
also $\phi(\ov W)=\phi(W)$, we have $\phi(\ov W\cat W')=\phi(W)+\phi(W')=0$.
That is, $\ov W\cat W'$ is an open parity walk from $u$ to $v$, so $\phi$ is
not a spec.

Conversely, if $\phi$ is not a spec, then $G$ has an open parity $u,v$-walk $W$
for some vertices $u$ and $v$.  Letting $P$ be an $r,u$-path, it follows that
$P\cat W$ is an $r,v$-walk.  We obtain $\phi(P)\in S_u$ and
$\phi(P\cat W)\in S_v$ with $\phi(P)=\phi(P\cat W)$, since $\phi(W)=0$.
\end{proof}

The transformation $\phi\st\ang{E(G)}\to\UU$ behaves nicely when $\phi$ is a
spec.

\begin{lemma}\label{cosets}
If $\phi$ is a spec, then $S_r$ is a subspace of $\UU$, and the sets $S_v$ for
$v\in V(G)$ are cosets of $S_r$ (all with size $|S_r|$).
\end{lemma}
\begin{proof}
The set $S_r$ consists of the images under $\phi$ of closed walks from $r$
to $r$; it certainly contains $0$.  Also the concatenation of any two closed
walks from $r$ is a closed walk from $r$, so $S_r$ is closed under vector
addition.  Since $\UU$ is binary, this suffices to make $S_r$ a subspace.

Consider $x\in S_v$; note that $x=\phi(W)$ for some $r,v$-walk $W$.  
We show that $S_v$ is the set of vectors that can be obtained by adding a
member of $S_r$ to $x$; that is, $S_v=x+S_r$.

For $y\in S_v$, let $Z$ be an $r,v$-walk with $y=\phi(Z)$.  Note that
$W\cat \ov Z$ is a closed walk from $r$.  Thus
$x+y=\phi(W)+\phi(Z)=\phi(W\cat \ov Z)\in S_r$.
That is, $y=x+(x+y)\in x+S_r$.

Conversely, if $w\in S_r$ and $Y$ is a closed walk from $r$ such that
$\phi(Y)=w$, then $Y\cat W$ is an $r,v$-walk, and 
$x+w=\phi(W)+\phi(Y)=\phi(Y\cat W)\in S_v$.
%
\end{proof}

In this setting, the family of cosets is actually a quotient space,
in which $S_r$ is the zero element.  To keep the presentation self-contained,
we include the argument for this family forming a binary vector space, which
will lead us to a canonical edge-coloring.  We may write the coset containing
a vector $x$ as $[x]$ or as $S_v$, where $v$ is a vertex such that $x\in S_v$.

\begin{lemma}\label{Svspace}
When $\phi$ is a spec, the cosets $\{S_v\st v\in V(G)\}$ form a binary vector
space, with addition defined by $S_u+S_v=\{x+y\st x\in S_u,y\in S_v\}$.
\end{lemma}
\begin{proof}
We show that $S_u+S_v$ is a coset.  That is, $x+y$ and $x'+y'$ differ by an
element of $S_r$ whenever $x,x'\in S_u$ and $y,y'\in S_v$.  Choose $r,u$-walks
$W$ and $W'$ such that $x=\phi(W)$ and $x'=\phi(W')$, and choose $r,v$-walks
$Z$ and $Z'$ such that $y=\phi(Z)$ and $y'=\phi(Z')$.  The walk
$W\cat \ov W'\cat Z\cat \ov Z'$ is a closed walk from $r$, so
$\phi(W)+\phi(W')+\phi(Z)+\phi(Z')\in S_r$.
\end{proof}

We now obtain a more precise version of Theorem~\ref{thm:canonical_intro}.

\begin{theorem}\label{mainq}
Given a spec $\phi$ of $G$ rooted at $r$, the edge-coloring $\phi^*$ that
assigns to each edge $uv\in E(G)$ the coset $[\phi(uv)]$ is a canonical
edge-coloring of $G$.  In other words, $\phi^*=\pi\circ\phi$, where
$\pi(x)=[x]$.  In particular, $\phi^*$ uses no more colors than $\phi$.
\end{theorem}
\begin{proof}
With the cosets forming a binary vector space as in Lemma~\ref{Svspace},
it suffices to prove $[\phi(uv)]=S_u+S_v$.
Choose $x\in S_u$ and $y\in S_v$ such that $x=\phi(W)$ and $y=\phi(Z)$, where
$W$ is an $r,u$-walk and $Z$ is an $r,v$-walk.  Since $W\cat uv\cat \ov Z$
is a closed walk from $r$, we have $x+\phi(uv)+y\in S_r$.
This yields $[\phi(uv)]=[x+y]=S_u+S_v$.

Since $\phi(uv)$ is a well-defined vector (of weight $1$) in $\UU$,
every color class under $\phi$ is mapped into a single color class
under $\phi^*$.  Multiple color clases under $\phi$ may be mapped into
a single class under $\phi^*$, but $\phi^*$ uses no more colors than $\phi$.
\end{proof}

We provide an example to show how the renaming and collapsing of colors
may occur in obtaining the canonical edge-coloring from a spec $\phi$.

\begin{example}
Consider the cycle $C_4$ with vertices $r,s,t,u$ in order, and
let $\phi$ assign the colors $1,2,1,3$ to the edges $rs,st,tu,ur$ in order.
In the atomic form, $\phi$ encodes these colors
as $100$, $010$, $100$, $001$ in $\UU$, respectively (we
can ignore all coordinates after the third).  Letting $r$ be the root,
we have $S_r=\{000,011\}$, $S_s=\{100,111\}$, $S_t=\{110,101\}$, and
$S_u=\{010,001\}$.  The construction in the proof produces $\phi^*$ with
$\phi^*(rs)=\{100,111\}$, $\phi^*(st)=\{010,001\}$, $\phi^*(tu)=\{100,111\}$,
and $\phi^*(ur)=\{010,001\}$.  Only two colors are used in $\phi^*$, and
$\phi^*$ is a canonical edge-coloring.  Relabeling $1$ as $\{100,111\}$
and both $2$ and $3$ as $\{010,001\}$ expresses $\phi$ as a refinement of
the canonical edge-coloring $\phi^*$.
\end{example}

As noted in the introduction, Theorem~\ref{mainq} completes the proof
of $\spec(K_{s,t})=s\circ t$.  It also enables us to directly answer
Question~\ref{q:log}.  Let $\im(\phi)$ denote the image of a map $\phi$.
Also write $\ang{S}$ for the vector space that is the span of a set $S$
of vectors.

\bigskip
\noindent
{\bf Theorem~\ref{lgn}.}
If $G$ is a connected $n$-vertex graph, then $\spec(G)= \lceil\lg n\rceil$
if and only if $G$ is a subgraph of the hypercube $Q_{\lceil \lg n \rceil}$.

\begin{proof}
Let $f$ be a labeling of $V(G)$ whose associated canonical edge-coloring $\phi$
is an optimal spec of $G$, as guaranteed by Theorem~\ref{main}.  For a fixed
vertex $z\in V(G)$, adding $f(z)$ to the image of each vertex under $f$ does
not change $\phi$, so we may assume $f(z)=0$.

By definition $\phi(uv)=f(u)+f(v)$ for each edge $uv$, and for each vertex $v$
we have $f(v)$ as the sum of the colors along a $z,v$-path.  Thus the sets of
colors under $f$ and $\phi$ have the same span; that is,
$\ang{\im(f)}=\ang{\im(\phi)}$.

Let $k = \lceil\lg n\rceil$. By Proposition~\ref{prop:subQk},  $\spec(G)\ge k$. 
If equality holds, then $|\im(\phi)|=k$.  Since $f$ is injective,
$2^{\dim\ang{\im(f)}}\geq n$.
Thus $\dim\ang{\im(\phi)} = \dim\ang{\im(f)} \geq k$.  Since $|\im(\phi)|=k$,
the vectors in $\im(\phi)$ must be linearly independent.
Choosing $\im(\phi)$ as a basis maps $\ang{\im(\phi)}$ to the binary vector
space $\FF_2^k$ taking the vectors in $\im(\phi)$ to those with exactly one
nonzero coordinate.  Since $\im(f)$ and $\im(\phi)$ have the same span,
this isomorphism also transforms $f$ into an injection from $V(G)$ to
$V(Q_k)$ that expresses $G$ as a subgraph of $Q_k$.
	
Conversely, if $G$ is a subgraph of $Q_k$, then viewing the vertices of
$Q_k$ as binary vectors and taking $f$ to be the identity map yields we get
$|\im(\phi)| \leq k$.
\end{proof}	

\section{Lower bounds for $\spec(G)$}\label{sec:lower}
In this section, we use the algebraic structure of the previous section
to obtain new lower bounds on the spec number of a graph.
Again we consider a connected $n$-vertex graph $G$ rooted at vertex $r$
and take an edge-coloring $\phi$ presented as an atomic edge-coloring.
We write $|\phi|$ for the number of colors used by $\phi$,
so $|\phi|=|\im(\phi)|$.

\begin{lemma} \label{lem:Sr}
If $\phi$ is a spec of an $n$-vertex connected graph $G$,
then $|\phi| \geq \lg n + \lg |S_r|$. 
\end{lemma}
\begin{proof}
By Lemma~\ref{cosets}, the cosets $S_v$ have the same size.
By Lemma \ref{lem:disjoint}, $S_u \cap S_v = \nul$ for distinct $u,v\in V(G)$.
The cosets all lie in the span of $\im(\phi)$.
Since $\phi$ is atomic, the size of the span is $2^{|\phi|}$.
These four statements yield the four steps in the following computation.
\[
n |S_r| = \sum_{v \in V(G)} |S_v| 
= \Bigl |  \bigcup_{v \in V(G)} S_v  \Bigr |
\leq |\langle \im(\phi) \rangle | = 2^{|\phi|}.
\]
Thus ${|\phi|} \geq \lg n + \lg |S_r|$. 
\end{proof}

Note that Lemma~\ref{lem:Sr} strengthens Proposition~\ref{prop:subQk}.
We use it to give another short proof of Theorem~\ref{lgn}.

\bigskip
\noindent
{\bf Theorem~\ref{lgn}.}
If $G$ is a connected graph on $n$ vertices, then $\spec(G)= \lceil\lg n\rceil$
if and only if $G$ is a subgraph of the hypercube $Q_{\lceil \lg n \rceil}$.
\begin{proof}
As always, $\spec(G)\ge p(G)\ge\CL{\lg n}$.  By Lemma~\ref{lem:Sr},
$|\phi|=\CL{\lg n}$ requires $S_r$ to contain only the $0$-vector.

Given a cycle $C$, let $P$ be a path from $r$ to a vertex of $C$.  The
concatenation $P\cat C\cat\ov P$ is a closed walk from $r$, so
$\phi(C)=\phi(P)+\phi(C)+\phi(P)\in S_r$.  Now $S_r=\{0\}$ requires
$C$ to be a parity walk.  Hence $G$ satisfies the characterization by
Havel and Mor\'avek~\cite{HM} mentioned in the introduction:
$G\esub Q_k$ if and only if $G$ has a pec $\phi$ using at most $k$ colors such
that every cycle is a parity walk.
\end{proof}

Next we develop a quantitative tool for lower bounds on $\spec(G)$ that is
based on the structure of $G$.

\begin{definition}
Given an ordering $\VEC v1n$ of the vertices of $G$, let $d^-(v_i)$ denote the
number of neighbors of $v_i$ among $\VEC v1{i-1}$; this is the
\emph{back-degree} of $v_i$ relative to the given ordering.  Vacuously 
let $d^-(v_1)=0$.  A subset $T$ of $V(G)$ is \emph{saturating} with respect to
the ordering $\VEC v1n$ if $|T \cap \{\VEC v2k\}| \geq \lg{k}$ for $k \leq n$.
\end{definition}
 
\begin{lemma}\label{l:lowerBound}\label{lem:density}
If $(\VEC v1n)$ is an ordering of the vertices of an $n$-vertex connected
graph $G$ such that each vertex after the first has an earlier neighbor, then
$$
\hat{p}(G) \geq
\min_{T\in\bT} \sum_{v \in T}d^-(v),
$$
where $\bT$ is the family of saturating subsets with respect to $\VEC v1n$.
\end{lemma}	

\begin{proof}
It suffices to show for every canonical edge-coloring $\phi$ that there is
at least one saturating set $T$ such that $|\phi|\ge\sum_{v\in T}d^-(v)$.
Let $f$ be an injective labeling of $V(G)$ that generates $\phi$.
Adding $f(v_1)$ to each $f(v_i)$ does not change the associated 
canonical edge-coloring $\phi$, so we may assume $f(v_1)=0$.

For $1\le i \le n$, let $G_i$ be the subgraph of $G$ induced by 
$\{\VEC v1i\}$.  Since each vertex after $v_1$ has an earlier neighbor,
each $G_i$ is connected.  Let $f_i$ be the restriction of $f$ to $V(G_i)$,
and let $\phi_i$ be the restriction of $\phi$ to $E(G_i)$.
Note that $\dim\ang{\im(f_1)}=0$.

Let $T$ be the set of vertices $v_i$ with $i\ge2$ such that
$f(v_i) \not \in \ang{\im(f_{i-1})}$.  Since $f(v_i)$ is a single vector,
$\dim\ang{\im(f_i)}= \dim\ang{\im(f_{i-1})}+1$ for $v_i \in T$, and
$\dim\ang{\im(f_i)}= \dim\ang{\im(f_{i-1})}$ for $v_i\notin T$.
For $k\le n$ this yields
$|T\cap\{\VEC v2k\}|= \dim\ang{\im(f_k)} \geq \lg{k}$, where the last
inequality holds because $f$ is injective.  Thus $T$ is saturating.

For $v_i \in T$ and $v_jv_i \in E(G)$ with $j<i$, we have
$\phi(v_jv_i) = f(v_j) + f(v_i) \notin \ang{\im(f_{i-1})}$.
Since $\im(\phi_{i-1})\esub \ang{\im(f_{i-1})}$ and $\phi$ is injective on
$E(G_i)-E(G_{i-1})$ (because $\phi$ is a proper edge-coloring), it follows that
$|\phi_i|-|\phi_{i-1}| = d^-(v_i)$.
Thus $|\phi|\ge\sum_{i\in I}d^-(v_i)$, as desired.
\end{proof}

\begin{corollary}\label{cor:density}
If $\phi$ is a spec for  $G$ with vertex ordering $\VEC v1n$ such that
$d^-(v_i)>0$ for $i\ge2$, then $|\phi|\ge L$, where $L$ is the least sum of
$\CL{\lg n}$ back-degrees.
\end{corollary}
\begin{proof}
Every saturating set has size at least $\CL{\lg n}$.
\end{proof}

\begin{remark}\label{cycimage}
Before leaving this section, we describe an alternative lower bound on
$\spec(G)$ using Lemma~\ref{lem:Sr}.  This tool is not as easy to prove or
apply as Lemma~\ref{lem:density}, but it sometimes gives stronger bounds,
as we will note in Remark~\ref{altbip}.

When we consider the cosets $S_v$ associated with an atomic spec $\phi$ on
a graph $G$ rooted at vertex $r$, the set $S_r$ is the nullspace of a
linear transformation.  It is easy to show that $S_r$ is the span of the set of
images under $\phi$ of cycles in $G$; we call this vector space $\cC_\phi$.
The equality $S_r=\cC_\phi$ is the essence of our second proof of
Theorem~\ref{lgn} above.  Since $\cC_\phi$ is a binary vector space, we can
view Lemma~\ref{lem:Sr} as stating
$$|\phi|\ge\lg n+\dim\cC_\phi,$$
and our task becomes finding lower bounds on $\dim\cC_\phi$. 
Again the back-degrees are relevant and yield a general lower bound that
again implies Corollary~\ref{cor:density}:

{\narrower\smallskip\noindent\it
Let $\phi$ be a spec of $G$ with vertex ordering $\VEC v1n$ such that
$d^-(v_i)\ge1$ for $2\le i\le n$.  If the back-degrees of any $t$ distinct
vertices (not including $v_1$) sum to more than $|\phi|$, then
$\dim \cC_\phi>|\phi|-t$.
\par\smallskip
}

\noindent
We briefly sketch the proof.  As in the proof of Lemma~\ref{lem:density},
consider the subgraphs $G_i$ and maps $f_i$ and $\phi_i$ in order.  Instead of
the indices $i$ such that $f(v_i)\notin \ang{\im(f_{i-1})}$, focus on the
indices $i$ such that $\im(\phi_i)\not\esub \im(\phi_{i-1})$.  If $k$
new colors occur on the edges from $v_i$ to earlier neighbors, then we add $k$
or at least $k-1$ vectors to a basis for $\cC_\phi$, depending on whether some
edge $v_jv_i$ with $j<i$ has a color already seen.  Hence
$\dim\cC_\phi\ge|\phi|-q$, where $q$ is the number of indices $i$ such that all
edges from $v_i$ to earlier neighbors have distinct new colors.  Since the
back-degrees of all such vertices sum to at most $|\phi|$, we have $q<t$.
\end{remark}

\section{Powers of paths}\label{sec:pathpower}

We apply Corollary \ref{cor:density} to obtain a lower bound for $\spec(G)$
when $G$ is a power of a path, and we will then obtain an upper bound that
is only slightly larger.  The \emph{$\ell$th power} of a graph $H$, written
$H^{\ell}$, is the graph with vertex set $V(H)$ in which distinct vertices $u$
and $v$ are adjacent if and only if the distance between $u$ and $v$ in $H$ is
at most $\ell$.

\begin{theorem} \label{thm:path}
If $n$ and $\ell$ are integers with $1\le \ell\le\CL{\lg n}$, then
\[\spec(P_n^\ell)> \ell \lceil \lg n \rceil - \binom{\ell + 1}{2}.\]
\end{theorem}
\begin{proof}
When the vertices of $P_n^\ell$ are placed in the same order $\VEC v1n$
as along the underlying path, we have $d^-(v_i)=i-1$ for $1\le i\le\ell$
and $d^-(v_i)=\ell$ for $\ell<i\le n$.  Hence for $k\ge\ell$ the least sum of
$k$ back-degrees is $\ell k-\binom{\ell+1}2$.  Set $k=\CL{\lg n}$ and apply
Corollary~\ref{cor:density}.
\end{proof}

We next describe a canonical edge-coloring for the path-power $P_n^\ell$
that establishs an almost tight upper bound on $\spec(P_n^\ell)$.  The coloring
is based on the standard Gray code for binary lists (\cite{Gil,RND}).
In this section we use the notation $\la c\ra$ for a sequence whose entries
have the form $c_i$.  Also let $[k]$ denote the set $\{1,\ldots,k\}$.

\begin{construction}\label{constrPnl}
Form a sequence $\la c\ra$ by assigning value $j$ to the positions
indexed by odd multiples of $2^{j-1}$.  Thus $c_i$ is $1$ plus the exponent on
the highest power of $2$ dividing $i$, and the sequence begins
$$1,2,1,3,1,2,1,4,1,2,1,3,1,2,1,5,1,2,\ldots.$$
We use the term {\it window} to describe a string $(\VEC c{i+1}j)$ of
consecutive entries in $\la c\ra$.

An equivalent description of $\la c\ra$ is useful for
many proofs.  Let $A_k$ denote the list of the first $2^{k-1}$ entries in 
$\la c\ra$.  Recursively, we start with $A_1$ being the single entry $1$
and for $k>1$ let $A_k$ be the concatenation $A_{k-1},k,A_{k-1}$.
Note that $k$ is the largest entry in $A_k$, the middle entry of $A_k$
is its only copy of $k$, and the reverse of $A_k$ is the same as $A_k$.

Let $e_j$ be the atom in $\UU$ having its $1$ in coordinate $j$ 
(the $j$th ``standard basis vector'').  Convert $\la c\ra$ to a sequence
$\la a\ra$ of atoms in $\UU$ by letting $a_i=e_{c_i}$.
Let $s_i=\SE j1i a_j$, with the sum taken in $\UU$; the empty sum $s_0$ is
the $0$-vector.  The sequence $s_0, s_1, \dots$ shows the elements of
$\UU$ in order as they are visited by the standard binary Gray code.
The coordinate $c_i$ is the coordinate that changes when moving from $s_{i-1}$
to $s_i$ in Gray code.  Thus $s_{i-1}+s_i=a_i$.

In this construction we index the vertices of $P_n^\ell$ as $\VEC v0{n-1}$.
Let $\phi$ be the canonical coloring of $P_n^\ell$ generated by assigning label
$s_i$ to vertex $v_i$ for $0\le i\le n-1$.  The color on an edge $v_iv_j$ is
thus $s_i+s_j$ when $1\le j-i\le\ell$, and this color equals $\VECOP a{i+1}j+$.
When $j>i$, the {\it length} of an edge $v_iv_j$ (or the corresponding
window $(\VEC c{i+1}j)$ in $\la c\ra$) is $j-i$.
We maintain the notation of this construction for the
remainder of this section.
\end{construction}

\begin{remark}\label{rem:colors}
Having defined the canonical coloring $\phi$ of $P_n^\ell$ using vectors in
$\UU$, we know that it is a spec, by Observation~\ref{obs:canon}.  With each
color being a vector in $\UU$, we can equivalently view the colors under $\phi$
as sets of nonnegative integers, where the color set corresponding to a vector
$x$ is the set of coordinates whose entry in $x$ is $1$.

By the definition of $\la c\ra$, the largest coordinate that can be nonzero
in the colors under $\phi$ is $\CL{\lg n}$, so we can view the colors as
subsets of the set $[\CL{\lg n}]$.  The color on $v_{j-1}v_j$ is then
$c_j$, and more generally the color on $v_iv_j$ is the set consisting of all
integers appearing an odd number of times in the window $(\VEC c{i+1}j)$.
We write this color also as $\phi(\VEC c{i+1}j)$.

The windows ending with the first appearance of $k$ at $c_{2^{k-1}}$ have $k$
as their largest element, appearing once in the window.  The windows
$(\VEC ci{2^{k-1}})$ as $i$ runs from $2^{k-1}-1$ to $0$ have colors
consisting of $k$ and the subsets of $[k-1]$ in Gray code order (since the
sequence of coordinate changes in the Gray code on $k-1$ coordinates is
unchanged by reversal).  Hence these edges provide $\min\{\ell,2^{k-1}\}$
colors on $E(P_n^\ell)$ in which $k$ is the largest element.
\end{remark}

We will show that the colors with largest element $k$ that are described
in Remark~\ref{rem:colors} are in fact all the colors under $\phi$ in which $k$
is the largest element.  We first obtain an elementary property of the 
windows in $\la c\ra$.

\begin{lemma}\label{morek}
In any window in $\la c\ra$, the largest element appears exactly once.
\end{lemma}
\begin{proof}
Let $W$ be a window in $\la c\ra$.  Since $A_1,A_2,\ldots$ are successively
longer initial segements of $\la c\ra$, we may let $k$ be the least $i$ such
that $A_i$ contains $W$.  Since $A_{k-1}$ does not contain $W$, the window
contains the element $k$ at the center of $A_k$.  This is the largest element
in $A_k$, appearing only once, and hence the same is true for $W$.
\end{proof}

Every appearance of $k$ in $\la c\ra$ is at a position that is an odd
multiple of $2^{k-1}$ and lies at the center of a window $W$ of length $2^k-1$
that is a copy of $A_k$.  Any window in which this copy of $k$ is the 
largest element is contained in $W$.  Hence we may restrict our
attention to windows in $A_k$ containing the central element $k$.

\begin{lemma}\label{smallk}
If $\ell\ge2^{k-1}$, then every color under $\phi$ arising from a window in
$\la c\ra$ in which $k$ is the largest element arises from a window ending at a
copy of $k$.
\end{lemma}
\begin{proof}
When $k$ is the largest element in a window in $\la c\ra$, the color on that
window consists of $k$ and some subset of $[k-1]$.  The binary sums of the
windows in $A_{k-1}$ that end with the last position of $A_{k-1}$ (together
with the empty window) yield all subsets of $[k-1]$ (in Gray code order).
Since $\ell\ge2^{k-1}$, under $\phi$ we obtain all those colors from windows
ending at position $2^{k-1}$, containing $k$.
\end{proof}

When $\ell<2^{k-1}$, the situation is more delicate.  We need the color on a
window with largest element $k$ and length at most $\ell$ to arise also from a
window ending at $k$ that is no longer than $\ell$, so that the color occurs
in the list of colors we have described for $\phi$.  To prove this we use the
next lemma.  Informally it states that if we obtain a window $W$ by trimming
about the same amount from the beginning and end of a window $W'$ starting at
$c_1$, then the color on $W$ is the color on some window starting at $c_1$ that
is no longer than $W$.

\begin{lemma}\label{lem:trim}
Let $\la a\ra$ be the sequence of atoms in $\UU$ obtained from the sequence
$\la c\ra$ of Construction~\ref{constrPnl}.  If $q,r,m$ are nonnegative
integers with $r\le m/2$ and $r\in\{q,q+1\}$, then
${a_{q+1}+\cdots+a_{m-r}}\in\{\VEC s0m\}$, where $s_i$ is the sum of the
first $i$ terms in $\la a\ra$.
\end{lemma}
\begin{proof}
Let $B= a_{q+1} + \dots + a_{m-r}$.  We use induction on $r$.

If $r=0$, then $q=0$ and $B = s_m$.  If $r=1$ and $q=0$, then $B=s_{m-1}$.
When $r=1$ and $q=1$, we consider two cases.
If $c_m = 1$, then $B = a_2 + \dots + a_{m-1} = a_1 + \dots + a_m = s_m$.
If $c_m \ne 1$, then $c_{m-1} = 1$, and $B = a_1+\dots+a_{m-2} = s_{m-2}$.

For the induction step, suppose $r \geq 2$.  Let $x$ count the appearances
of $1$ in the list $(\VEC c1q)$, so $x = \CL{q/2}$.  Similarly, let $y$ count
the appearances of $1$ in $(\VEC c{m-r+1}m)$, so $y$ is $\CL{r/2}$ when $m$
is odd and is $\FL{r/2}$ when $m$ is even.  Let $q' = q-x = \FL{q/2}$ and
$r' = r-y$, so $r'\in\{\FL{r/2},\CL{r/2}\}$.  We will use an instance of the
claim with $q'$ and $r'$ in place of $q$ and $r$.  Clearly $q'$ and $r'$ are
nonnegative.

Since $y\le\CL{r/2}$, we have
$r'=r-y\ge\FL{r/2}\ge\FL{q/2} = q'$.  If $r=q$ or $q$ is even, then
$y\ge\FL{r/2}$ yields $r'=r-y\le\CL{r/2}\le\FL{r/2}+1\le\FL{q/2}+1=q'+1$.
In the remaining case, $r = q+1$ and $q$ is odd, so $y=r/2$, and
$r'=r/2=\FL{q/2}+1=q'+1$.  Thus $r'\in\{q',q'+1\}$.  Also, since $r\ge2$
we have $y\ge\FL{r/2}\ge1$ and $r'<r$.

We will apply the induction hypothesis to the instance $(q',r',m')$
with $m'=\FL{m/2}$.  We need $r'\le m'/2$.  If $m$ is odd, then
$y=\CL{r/2}$ and $m'=(m-1)/2$, and we have
$r'=\FL{r/2}\le\FL{(m-1)/4}=\FL{m'/2}$.  If $m$ is even,
then $y=\FL{r/2}$ and $m'=m/2$, and we have
$r'=\CL{r/2}\le\CL{m/4}=\CL{m'/2}$.  This satisfies the criteria for the
parameters unless $m$ is an odd multiple of $2$ and $r=m/2$.  Since
$q\in\{r,r-1\}$, in that special case $B\in\{s_0,s_1\}$ and the claim holds.
Hence we may assume $r'\le m'/2$, and the induction hypothesis applies.

With $B'=\VECOP a{q'+1}{m'-r'}+$, the induction hypothesis yields
$B'\in \{\VEC s0{m'}\}$.  Let $B'=s_h$ with $h\le m'$.
We next compare $B$ and $B'$.

Let $\la d\ra$ be the sequence obtained from $\la c\ra$ by deleting all
copies of $1$ and then reducing all remaining entries by $1$.
The sequence $\la d\ra$ has $j$ in position $i$ when $\la c\ra$ has $j+1$ in 
position $2i$, which occurs when $i$ is an odd multiple of $2^{j-1}$,
so both $d_i=c_{2i}-1$ and $\la d\ra = \la c\ra$ are true.
This also tells us that $s_{2i}$ is obtained from $s_i$ by shifting the
position of each nonzero coordinate up by $1$ and then adding $e_1$ if and only
if $i$ is odd.

In order to compare $B'$ with $B$, let $f$ be the linear map on $\UU$ that maps
$e_1$ to the $0$-vector and maps $e_j$ to $e_{j-1}$ for $j\ge2$.
We claim $B'=f(\VECOP a{q+1}{m-r}+)=f(B)$, because the contribution from
odd-indexed terms in the sum $B$ is ignored.  More precisely, the first
contribution to $B'$ is from $a_{2q'+2}$, which is $a_{q+1}$ when $q$ is odd
and is $a_{q+2}$ when $q$ is even.  When $q$ is even, the contribution from
$a_{q+1}$ to $B$ is $e_1$ and is ignored under $f$.  Similarly, the last
even-indexed contribution is from $a_{2(m'-r')}$, which is $a_{m-r}$ when $m$
and $r$ have the same parity and is $a_{m-r-1}$ when they have opposite parity.
The case of opposite parity is no problem because $f$ then ignores
the contribution of $e_1$ from $a_{m-r}$.

Thus the even-indexed terms in $B$ correspond to the terms in $B'$, with
the nonzero coordinate positions for such a term in $B$ augmented by $1$
in comparison to $B'$.  Their sum in $B$ is $s_{2h}$ if $h$ is even, and it is
$s_{2h}+e_1$ if $h$ is odd.  However, the number of copies of $e_1$ in the sum
$B$ may still have either parity.  Thus the value of $B$ may still be $s_{2h}$
or $s_{2h}+e_1$.  Since $2h$ is even, $s_{2h}+e_1=s_{2h+1}$.  Therefore,
we obtain $B\in\{\VEC s0m\}$ if $2h+1\le m$.  With $h\le m'=\FL{m/2}$, our only
worry is when $m$ is even, $h=m/2$, and $B=s_m+e_1$.

It thus suffices to show that $B$ and $s_m$ cannot differ by $e_1$ when
$m$ is even.  Note that the weight of the color on any window (the number of
nonzero coordinates in the corresponding vector in $\UU$) has the same parity
as the length of the window.  Since $m$ is even, $s_m$ has even weight.  If
$r=q$, then also $B$ has even weight.

We may therefore assume $r=q+1$.  Since $m$ is even, $a_m\ne e_1$, so
$y=\FL{r/2}$.  With $x=\CL{q/2}$ and $r=q+1$, we have $x=y$.  Hence the numbers
of $1$s in $(\VEC c1m)$ and $(\VEC c{q+1}{m-r})$ have the same parity.  Thus
again $B$ and $s_m$ cannot differ by $e_1$, and the problematic case does not
arise.
\end{proof}

\begin{lemma}\label{lem:differences}
If $k$ is the largest entry in a window $W$ in $\la c\ra$ corresponding to an
edge $v_iv_j$ in $P_n^\ell$, then the color $\phi(v_iv_j)$ is also the color
assigned to some edge $v_{i'}v_{j'}$ with $j'=2^{k-1}$ and $j'-\ell\le i'<j'$.
(The window corresponding to $v_{i'}v_{j'}$ has last entry $k$.)
\end{lemma}
\begin{proof}
By Lemma~\ref{smallk}, we may assume $\ell<2^{k-1}$, so the length of
$W$ is less than $2^{k-1}$.  As noted before Lemma~\ref{smallk},
we may assume that $W$ is contained in $A_k$ and is the window
$(\VEC c{i+1}j)$ with $i<2^{k-1}<j<2^k$.  By symmetry, we may assume
$2^{k-1}-i\le j-2^{k-1}$.

Recall that we have also expressed the first $2^k-1$ entries in $\la c\ra$
as $A_k$, where $A_k$ is the concatenation $A_{k-1},k,A_{k-1}$.  Inductively,
$A_k$ is unchanged by reversal, so the initial portion of the second copy of
$A_{k-1}$ is the reversal of the window $(\VEC c{i+1}{2^{k-1}-1})$ that ends
the first copy of $A_{k-1}$.

Let $q=2^{k-1}-1-i$, and let $m=j+q-2^{k-1}$.  Since $j-i\le\ell<2^{k-1}$,
we have $m=j-i-1<2^{k-1}-1=|A_{k-1}|$.  After canceling the contributions
from the windows of length $q$ in each direction from position $2^{k-1}$ and
indexing the remaining positions within $A_{k-1}$, we have
$\phi(W)=e_k+(\VECOP a{q+1}{m-q}+)$.  By Lemma~\ref{lem:trim},
$\VECOP a{q+1}{m-q}+\in\{\VEC s0m\}$.  Since $m<2^{k-1}$ and the copy
of $A_{k-1}$ preceding position $2^{k-1}$ generates Gray code reading from
the right, we have $\phi(W)$ equal to the color on a window of length at
most $\ell$ ending at position $2^{k-1}$.
\end{proof}

\begin{theorem}\label{pathupper}
For integers $n$ and $\ell$ with $1\le\ell\le\CL{\lg n}$,
$$\spec(P_n^{\ell}) < \ell\CL{\lg n}-\ell(\FL{\lg\ell}-1).  $$
\end{theorem}
\begin{proof}
We give to $P_n^\ell$ the canonical coloring $\phi$ based on Gray code as
described in Construction~\ref{constrPnl}.  By
Lemmas~\ref{morek}--\ref{lem:differences},
there are exactly $\min\{\ell,2^{k-1}\}$ colors in which $k$ is the largest
nonzero coordinate, for $1\le k\le\CL{\lg n}$, and these are all the colors.
Hence
\begin{align*}
\spec(P_n^\ell)
&\le\SE k1{1+\FL{\lg\ell}} 2^{k-1}+\SE k{2+\FL{\lg\ell}}{\CL{\lg n}}\ell\\
&=\ell(\CL{\lg n}-1-\FL{\lg\ell})+2^{1+\FL{\lg\ell}}-1\\
&\le \ell\CL{\lg n}-\ell\FL{\lg\ell}+\ell-1.
\end{align*}

\vspace{-2pc}
\end{proof}

\medskip
Our upper and lower bounds on $\spec(P_n^\ell)$ differ by a quadratic
function of $\ell$, independent of $n$.  We conjecture that the upper
bound is the truth.

\begin{conjecture}
The canonical edge-coloring $\phi$ defined using Gray code in
Construction~\ref{constrPnl} is an optimal spec of $P_n^\ell$.
\end{conjecture}

\section{Bipartite Graphs with Large $\spec(G)$}
In this section we use the tools of Sections~\ref{sec:lower}
and~\ref{sec:pathpower} to construct bipartite graphs $G$ with $\spec(G)>p(G)$.

\begin{lemma}\label{lem:cut-edge}
If $G_1$ and $G_2$ are the components of $G-e$ in a connected graph having
a cut-edge $e$, then $p(G) \leq 1 + \max \{p(G_1), p(G_2)\}$.
\end{lemma}
\begin{proof}
Let $k=\max \{p(G_1), p(G_2)\}$, and let $\phi$ be an edge-coloring of $G$ that
assigns a pec to $G_1$ using colors $1,\dots,k$, assigns a pec to $G_2$ using
colors $1,\dots,k$, and assigns $\phi(e) = k+1$.  Any path containing $e$ is
not a parity path, since $e$ is the only edge of color $k+1$.  Any other path
is confined to $G_1$ or $G_2$, on which $\phi$ restricts to a pec.  Hence
$\phi$ is a pec on $G$.
\end{proof}

We next use Lemma~\ref{lem:density} to disprove Conjecture~\ref{conj:bip}.
\begin{theorem}\label{thm:ex}
For each $n$ that is an even power of $2$, there is a bipartite graph $G$ such
that $p(G) = n+1$ and $\spec(G)> n-\lg n+(\lg n)^2/4$.
\end{theorem}
\begin{proof}
Form $G$ from the disjoint union of $K_{n,n}$ and a path $P$ with $2^n$
vertices by adding a cut-edge $e$ consist of a vertex of $K_{n,n}$ and an
endpoint of $P$.  By Proposition~\ref{lgn} and Lemma~\ref{lem:cut-edge},
$p(G) = n+1$.

Next we consider the lower bound on $\spec(G)$.  Let $(X,Y)$ be the bipartition
of $K_{n,n}$, with $X=\{\VEC x1n\}$ and $Y=\{\VEC y1n\}$.  Let $\VEC u1{2^n}$
be the vertices of $P$ in order, with $v_1$ being the endpoint of $e$ in $P$.
Order $V(G)$ as $\VEC v1{2n+2^n}$ with $v_{2j-1}=x_j$ and $v_{2j}=y_j$ for
$1\le j\le n$, followed by $\VEC u1{2^n}$ in order.  With this ordering,
$d^-(v_i)=\lfloor i/2 \rfloor$ for $1\le i \le 2n$.  For even $t$, the sum of
any $t$ backdegrees is at least $2\SE i1{t/2} i$, which exceeds $t^2/4$.

Let $I$ be a saturating set.  Thus $|I\cap (X\cup Y)|\ge 1+\lg n$, and
the sum of the back-degrees of these vertices exceeds $(\lg n)^2/4$.
Also, $|I|$ is at least $\CL{\lg(2^n+2n)}$, which is $n+1$.  Each member of $I$
has back-degree at least $1$, and we have already accounted for the 
back-degrees of $1+\lg n$ vertices.  By Lemma~\ref{lem:density}, we
obtain $\spec(G)> n-\lg n+(\lg n)^2/4$.
\end{proof}

\begin{remark}\label{altbip}
Our purpose in Theorem~\ref{thm:ex} was only to give a quick counterexample
to Conjecture~\ref{conj:bip}.  Because of this and the stronger example in
Construction~\ref{constr}, we did not compute $\spec(G)$ carefully.  In fact,
a stronger lower bound can be obtained by using the lower bound in
Remark~\ref{cycimage} and setting $t=2\CL{\sqrt{|\phi|}}$ in the analysis
of back-degrees in $K_{n,n}$.  With this and an explicit canonical
edge-coloring, for the graph $G$ of Theorem~\ref{thm:ex} we can prove
\[2n - (2+o(1))\sqrt n \leq  \spec(G) \leq 2 n - \lg n+ 1.\]
We omit the details of these bounds.
\end{remark}

Building on the intuition behind the construction in Theorem~\ref{thm:ex},
we next present a more intricate construction to show that the ratio
$\spec(G)/p(G)$ can be arbitrarily large.  More precisely, $\spec(G)$ can be
$\Omega(k\lg k)$ when $p(G)=k$, even when restricted to bipartite graphs.
The main idea, used also in Theorem~\ref{thm:ex}, is that although a pec can
reuse colors without restriction on the two components of $G-e$ when $e$ is a
cut-edge receiving a new color, this fails for a spec.  The reason is that a
set of colors appearing on a cycle in one component of $G-e$ forbids that set
from appearing on a path in the other component of $G-e$, so a spec cannot
simply reuse colors in the two components.
Our construction will exploit this by having many cut-edges.
For each large integer $k$, we construct a bipartite graph $G_k$
such that $p(G_k) \leq 2k+o(k)$ and $\spec(G_k) \ge\s{\frac13-o(1)}{k\ln k}$.

\begin{construction}\label{constr}
We first introduce a bipartite analogue of the path power $P_n^\ell$ studied
in Theorem~\ref{thm:path}.  Given $n,\ell \in \NN$, let $\hat P(n,\ell)$ be the
graph with vertex set $\VEC v1n$ in which $v_i v_j$ is an edge if and only if
$|i-j| \le 2\ell$ and $i \not \equiv j\!\pmod 2$.

Next let $r = \lfloor k^{1/3} \rfloor$.  For $1\le \ell\le r$, let $Z_{\ell}$
be a copy of $\hat P(2^{\lceil k/ \ell \rceil }, \ell)$, and let $Z'_\ell$
be obtained from $Z_\ell$ by adding a vertex $u$ adjacent only to $v_1$.
Let $G_k=\bigcup_{\ell=1}^r Z'_\ell$.  In taking the union, we
use the same vertex $u$ for all $\ell$, so $u$ has degree $r$ in $G_k$.

Observe that $\hat P(n,\ell)\esub P_n^{2\ell}$.
Theorem~\ref{pathupper} now yields
$$
p(\hat P(2^{\CL{k/\ell}},\ell))
\le\spec(P_{2^{\CL{k/\ell}}}^{2\ell})
<2\ell(\CL{k/\ell}-\CL{\lg2\ell}+1)\le2\ell(k/\ell)=2k.
$$
Lemma~\ref{lem:cut-edge} now yields $p(G_k)< 2k+k^{1/3}$.
\end{construction}

We use Lemma~\ref{l:lowerBound} to prove a lower bound on $\spec(G_k)$.
The method of Remark~\ref{cycimage} yields a similar conclusion using
a lot more numerical work.

\begin{theorem}
For the graph $G_k$ in Construction~\ref{constr},
$$	\hat{p}(G) \geq \s{\frac{1}{3}-o(1)}k \ln{k}.$$

\end{theorem}
\begin{proof}
Order the vertices of $G_k$ starting with the universal vertex $u$ followed by
the vertices of subgraphs $\VEC Zr1$ in order, with the vertices of $Z_\ell$
placed in the natural ordering of vertices in the definition of
$\hat{P}(2^{\CL{k/\ell}}, \ell)$.
	
Let $n = |V(G_k)|$, and let $T$ be a saturating set of vertices in $G_k$ with
respect to this order.  By Lemma~\ref{l:lowerBound} it suffices to show 
$\sum_{v \in T}d^-(v) \geq \left(\frac{1}{3}-o(1)\right)k\ln{k}.$

For $1\le d\le r$, let $Y_d=\UE \ell dr V(Z_\ell )$ and
$X_d =\{v\st d^-(v)\ge d\}$.  Note that $\{u\}\cup Y_d$ occupies an initial
segment of the vertex order and that $X_d\esub Y_d$.  Indeed, $X_d$ consists of
all of $Y_d$ except for $2d-2$ vertices at the beginning of each portion of
$Y_d$ occupied by the vertices of one subgraph $Z_\ell$.

Since $T$ is a saturating set, $|T\cap Y_d|\ge\lg|Y_d|$.  In addition,
$|V(Z_\ell)|$ decreases rapidly as $\ell$ increases, so we do not lose much
by using $|Z_d|$ as a lower bound on $|Y_d|$.  In particular,
$$
|T\cap X_d|\ge \lg|Y_d|-(r-d+1)(2d-2)
\ge \CL{\frac kd}-(r-d+1)(2d-2) \geq \frac{k}{d}-2rd.
$$

A vertex in $T\cap X_d$ with back-degree $d$ also lies in each of
$\VEC X1{d-1}$ (and not in $X_{d+1}$).  Hence we can alternatively compute
$\sum_{v\in T}d^-(v)$ by the sum $\SE d1r|T\cap X_d|$.  This yields
\begin{align*}
\sum_{v \in T}d^-(v)
&= \sum_{d=1}^{r}|T \cap X_d| \geq \sum_{d=1}^r\s{ \frac{k}{d}-2rd} \\
&\geq  \sum_{d=1}^r \frac{k}{d} - r^2(r+1)
\geq k \ln\s{\lfloor k^{1/3}\rfloor} - k - k^{2/3} \\
&= \left(\frac{1}{3}-o(1)\right)k \ln{k},
\end{align*}
as desired.
\end{proof}

We conclude that when $k$ is sufficiently large, $p(G_k) = O(k)$, while
$\spec(G_k) = \Omega(k \ln k)$.  It remains unknown how large $\spec(G)$
can be as a function of $p(G)$.

\section*{Acknowledgment}
  
This research was initiated at the 2023 Montreal Graph Theory Workshop.
We thank the organisers and participants for creating a stimulating working
environment.

\end{document}